\documentclass{article} % use \documentstyle for old LaTeX compilers
\usepackage{arxiv}
\usepackage{pdflscape}
\usepackage[english]{babel} % 'french', 'german', 'spanish', 'danish', etc.
\usepackage{amssymb}
\usepackage{amsmath}
\usepackage{amsthm}
\usepackage{mathtools}
\usepackage{algorithm, algorithmic}
\usepackage{abstract}
\usepackage{pgfplotstable}
\usepackage[colorlinks=true,
            linkcolor=blue,
            citecolor=blue,
            urlcolor=blue]{hyperref}

\usepackage{xcolor}
\usepackage{breqn,url}
\usepackage{tikz}
\usepackage{pgfplots}
\usepackage{graphicx}
\usepackage{caption}
\usepackage{subcaption}

\usepackage{pgfplots}
\usepackage{pgfplotstable}

\usepackage{lipsum}
\usepackage{amsfonts}
\usepackage{graphicx}
\usepackage{epstopdf}
\usepackage{algorithmic}

\usepackage{todonotes}
\usepackage{enumerate}

\usepackage{amssymb}
\usepackage{mathtools}
\usepackage{booktabs}
\newcommand{\vertiii}[1]{{\left\vert\kern-0.25ex\left\vert\kern-0.25ex\left\vert #1
    \right\vert\kern-0.25ex\right\vert\kern-0.25ex\right\vert}}

\ifpdf
  \DeclareGraphicsExtensions{.eps,.pdf,.png,.jpg}
\else
  \DeclareGraphicsExtensions{.eps}
\fi

\newtheorem{dfn}{Definition}

\newtheorem{lemma}{Lemma}

\newtheorem{example}{Example}

\newcommand{\R}{{\mathbb R}}

\newtheorem{theorem}{Theorem}
\newtheorem{corollary}{Corollary}

\title{On the convergence rate of  the boosted Difference-of-Convex Algorithm (DCA)}

\author{
\textbf{Hadi Abbaszadehpeivasti}
	\thanks{Tilburg University,	 h.peivasti@tilburguniversity.edu   }
	\and
	\textbf{Etienne de Klerk}
	\thanks{Tilburg University,	e.deklerk@tilburguniversity.edu
	}	
	\and
\textbf{Adrien Taylor}
	\thanks{INRIA,	adrien.taylor@inria.fr }
}

\renewcommand{\leq}{\leqslant}
\renewcommand{\geq}{\geqslant}

\begin{document}
	\maketitle
	
	\begin{abstract}
	The difference-of-convex algorithm (DCA) is a well-established nonlinear programming technique
that solves successive convex optimization problems. These sub-problems are obtained from the difference-of-convex~(DC)
decompositions of the objective and constraint functions. We investigate the worst-case performance of the unconstrained DCA,
with and without boosting, where boosting simply performs an additional step in the direction generated by the usual DCA method.
 We show that, for certain classes of DC decompositions, the boosted DCA is provably better in the worst-case than the usual DCA.
 While several numerical studies have reported that boosted DCA outperforms classical DCA, a theoretical explanation
 for this behavior has, to the best of our knowledge, not been given until now.
 Our proof technique relies on semidefinite programming (SDP) performance estimation.

	\end{abstract}
	\section{Introduction}
We consider the unconstrained difference-of-convex (DC) optimization problem
	\begin{align}\label{P}
\  \ f^\star := &\inf f(x)\\
\nonumber & \ \text{s.t.} \ x\in \mathbb{R}^n
\end{align}
where $f: \mathbb{R}^n\to \mathbb{R}$ is a \emph{difference-of-convex~(DC)} function, i.e.\
$
f=f_1-f_2,
$
where
 $f_1$ and $f_2$ are closed convex proper functions. We also assume that the infimum $f^\star$ is finite, but not necessarily attained.
We will further restrict ourselves to the case where~$f_1$ and~$f_2$ are so-called functions of bounded curvature. To this end, recall that
a  function~$f$ has a \em{maximum curvature} $0 \le L < \infty$ if
 $
 x \mapsto \tfrac{L}{2}\| x\|^2-f(x)$  is convex,
and   \em{minimum curvature}
 $\mu > -\infty$ if
 $
 x \mapsto f(x)-\tfrac{\mu}{2}\| x\|^2$ { is convex.}
 The class of functions with minimum curvature $\mu$ and maximum curvature $L$ will be denoted by $\mathcal{F}_{\mu,L}(\mathbb{R}^n)$.

The celebrated difference-of-convex (DCA) algorithm for problem \eqref{P} may be stated as follows.

\begin{algorithm}[H]
\caption{DCA}
\begin{algorithmic}\label{Alg1}
\STATE Pick $x^1\in\mathbb{R}^n$ (starting point), $N \in \mathbb{N}$ (number of iterations).
\STATE For $k=1, 2, \ldots, N$ perform the following steps:\\
\begin{enumerate}
\item Pick  $g_2^{k}\in\partial f_2(x^{k})$.
\item
Set
\begin{align*}
  x^{k+1}\in \text{argmin}_{x\in\mathbb{R}^n} f_1(x)-\left(f_2(x^k)+\langle g_2^k, x-x^k\rangle\right).
\end{align*}
\end{enumerate}
\end{algorithmic}
\end{algorithm}
In the statement of Algorithm \ref{Alg1}, $\partial f_i(x^k)$ ($i=1,2$) refers to the set of sub-gradients of $f_i$ at $x^k$.
{ In this paper, we will only present new results for the case where $f_1$ and $f_2$ both have finite maximum curvature, which implies that both functions are differentiable and $L$-smooth. However, we will develop the framework of our analysis in the general setting, and therefore will refer to sub-gradients when we do not assume smoothness.}

The first convergence results for Algorithm \ref{Alg1} are attributed to \cite[Theorem 3(iv)]{tao1997convex}, namely if the sequence of iterates $\{x^k\}$ is bounded, then each accumulation point $x^\star$ of this sequence  is a critical
 point of problem \eqref{P}, i.e.\ $\partial f_1(x^\star)\cap  \partial f_2(x^\star) \neq \emptyset$.
For the history and properties of the DCA we refer to the surveys by  Le Thi and Dinh \cite{le2018dc,le2024survey} and Lipp and Boyd \cite{Boyde}.

A more recent variant of the DCA, the so-called boosted DCA \cite{Fukushima1981,Boosted_DCA_smooth_MP,Boosted_DCA_nonsmooth_SIOPT}, extends the DCA step as follows.

\begin{algorithm}[H]
\caption{Boosted DCA}
\begin{algorithmic}\label{Alg:BCDA}
\STATE Pick $x^1\in\mathbb{R}^n$ (starting point), $N \in \mathbb{N}$ (number of iterations), and $\alpha \in [0,1]$ (boosted step length).
\STATE For $k=1, 2, \ldots, N$ perform the following steps:\\
\begin{enumerate}
\item Pick  $g_2^{k}\in\partial f_2(x^{k})$.
\item
Set
\begin{align*}
  y^{k}\in \text{argmin}_{x\in\mathbb{R}^n} f_1(x)-\left(f_2(x^k)+\langle g_2^k, x-x^k\rangle\right).
\end{align*}
\item
Set $d^k = y^k - x^k$, and
\begin{align*}
  x^{k+1} = y^k + \alpha d^k.
\end{align*}
\end{enumerate}
\end{algorithmic}
\end{algorithm}
Note that the case $\alpha = 0$ corresponds to the usual DCA. In both Algorithms \ref{Alg1} and \ref{Alg:BCDA} we assume that { a subgradient of $f_2$ is} available at each iteration. The additional computational burden per iteration for boosted DCA is usually negligible compared to that of solving the intermediary subproblems and that of computing subgradients. It is therefore natural to compare the worst-case performance of the two methods on the same number of iterations.

The boosted DCA is motivated by the following lemma.
\begin{lemma}[Proposition 4 in \cite{Boosted_DCA_smooth_MP}]
Assume that $f$ is continuously differentiable, and  $f_1$ and $f_2$ have minimum curvature $\mu > 0$. Then, with reference to Algorithm \ref{Alg:BCDA}, one has
\[
\langle \nabla f(y^k), d^k\rangle \le - \mu \|d^k\|^2.
\]
\end{lemma}
In other words, if $f$ is smooth, and $f_1,f_2$ strongly convex, it is possible to reduce $f$ further along the DCA direction~$d^k$, i.e.\
$d^k$ is a descent direction at $y^k$.
This work investigates how different performance measures depend on the choice of ``boosted" step length $\alpha$.
Our goal is to give a theoretical explanation of the advantage that one may obtain using boosted DCA as opposed to DCA.
This, in turn, is motivated by successful numerical experiments in \cite{Boosted_DCA_smooth_MP,Zhang_et_al_2024_JOTA},
which show that boosted DCA typically outperforms DCA in practice.

{
In the literature, BDCA often involves some sort of line search.
We will primarily deal with the case that the boosted step length $\alpha$ is a fixed constant, but we will also show how our analysis may be adapted to deal with some backtracking line search variants}.

\subsection*{Contributions of this paper}
This work shows that boosted DCA enjoys the following worst-case convergence guarantee.

\begin{theorem}
\label{thm:main}
Let $f_1\in\mathcal{F}_{\mu, L}({\mathbb{R}^n})$ and $f_2\in\mathcal{F}_{\mu, L}({\mathbb{R}^n})$, $0 \le \mu < L < \infty$. After $N$ iterations of boosted DCA (Algorithm \ref{Alg:BCDA}), one has,
 \[
 { \min_{1\leq k\leq N+1}\frac{\left\|\nabla f(x^k)\right\|^2}{L} \equiv}
    \min_{1\leq k\leq N+1}\frac{\left\|g_1^k-g_2^k\right\|^2}{L}\leq \frac{(f(x^1)-f^\star) }{(1+\kappa\alpha)N + \frac{1}{2(1-\kappa)}},
\]
where $\kappa := \mu/L \in [0,1)$, and where $0 \le \alpha \le \min \{1,2\kappa\}$.
\end{theorem}
When $\alpha = 0$, this result reduces to a known (tight) convergence result for DCA, namely a special case of Theorem 3.1 in \cite{abbaszad2024DCA}. As soon as $\kappa > 0$, Theorem~\ref{thm:main} shows that improved performance guarantees can be obtained with $\alpha > 0$. { On the other hand, when $\mu = 0$, one has $\kappa = 0$, and Theorem \ref{thm:main} does not offer any improvement for positive $\alpha$. Thus, we only show improvement from the boosted step if $f_1$ and $f_2$ are strongly convex.}

{ 
The bound in Theorem~\ref{thm:main} can be tight, as the following example shows.

\begin{example}
\label{ex:tight}
    We will construct a univariate, piecewise quadratic function where the bound in Theorem~\ref{thm:main} holds with equality for the parameter choices:
\[
\mu = 1, \; L = 2, \; \kappa = \frac{1}{2}, \; \alpha = 1, \mbox{ and } N = 1.
\]
To this end, define the points
\[
x^1 = 0, \; y^1 = \frac{-1}{\sqrt{5}}, \;
x^2 = \frac{-2}{\sqrt{5}}, \; x_\star  =\frac{-4}{\sqrt{5}}. 
\]
We will construct functions $f_1 \in\mathcal{F}_{\mu, L}({\mathbb{R}})$ and $f_2 \in \mathcal{F}_{\mu, L}({\mathbb{R}})$ such that $x_\star$ is the minimizer of $f = f_1 - f_2$ on $\R$, $f(x_\star) = 0$, and the values $x^2$ and $y^1$ are generated by Algorithm \ref{Alg:BCDA} from the starting point $x^1$, where $f(x^1) = 1$. 
Define the piecewise quadratic function,
\[
f_1(x):=
\begin{cases}
\frac{1}{2}x^2 - \frac{1}{\sqrt{5}}x - \frac{12}{5}, & x\le x^\star,\\[1mm]
x^2 + \frac{3}{\sqrt{5}}x - \frac{4}{5}, & x^\star\le x\le x^{2},\\[1mm]
\frac{1}{2}x^2 + \frac{1}{\sqrt{5}}x - \frac{6}{5}, & x^{2}\le x\le y^1,\ \ \\[1mm]
x^2 + \frac{2}{\sqrt{5}}x - \frac{11}{10}, & y^1\le x,\ \ \\
\end{cases}
\]
and similarly let
\[
f_2(x):=
\begin{cases}
\frac{1}{2}x^2 - \frac{1}{\sqrt{5}}x - \frac{12}{5}, &  x\le x^{2},\\[1mm]
x^2 + \frac{1}{\sqrt{5}}x - 2, & x^{2}\le x\le y^1,\ \ \\[1mm]
\frac{1}{2}x^2  - \frac{21}{10}, & y^1\le x.\ \ \\
\end{cases}
\]
One may readily verify that $f_1$ and $f_2$ are continuously differentiable. Both derivatives are monotonically increasing (piecewise linear) functions, and therefore $f_1$ and $f_2$ are convex.
By the definition of curvature, it is also clear that both functions have minimum curvature $\mu =1$, and maximum curvature $L=2$, as required. Note also that $f(x^1) = f(0) = f_1(0) - f_2(0) = 1$, and that the minimum is $f(x_\star) = 0$.

We now perform one iteration ($N=1$) of Algorithm \ref{Alg:BCDA} on $f = f_1 - f_2$, starting from $x^1 =0$. Since $g_2^1 = \nabla f_2(x^1) = 0$, step 2 of the algorithm reduces to
\[
y^1 = \arg\min f_1(x) = -\frac{1}{\sqrt{5}},
\]
as required.
Thus the boosted step with $\alpha = 1$ yields, $x^2 = -\frac{2}{\sqrt{5}}$, as required.

It remains to verify that the bound in Theorem~\ref{thm:main} holds with equality for this example. To this end, note that 
\[
\min_{1\leq k\leq N+1}\frac{\left\|\nabla f(x^k)\right\|^2}{L} = \frac{1}{2}\min \left\{\frac{4}{5}, \frac{4}{5}\right\} = \frac{2}{5} = \frac{(f(x^1)-f^\star) }{(1+\kappa\alpha)N + \frac{1}{2(1-\kappa)}},
\]
as required.
\end{example}
}

\subsection*{Outline of this paper}
Theorem~\ref{thm:main} is proved below in Section \ref{sec:pep results} using the semidefinite programming (SDP) performance estimation approach, as pioneered by Drori and Teboulle \cite{drori2014performance}. More precisely, we use the refined version from~\cite{taylor2017smooth, Taylor} to allow for tight analysis. Thus, we will first review the performance estimation method in some detail in Section \ref{sec:PEP}.
{ In Section \ref{sec:backtracking} we show how one may use the analysis from Section \ref{sec:pep results} to analyze a variant of boosted DCA with  backtracking, as opposed to a fixed step length $\alpha > 0$. The advantage of the backtracking scheme is that one does not need to know the curvature parameters to select the step length.}
One of the attractive features of the DCA is that it contains some well-known algorithms as special cases, like gradient descent, and proximal gradient descent; see e.g.\ \cite{rotaru2025tight}. In Section \ref{sec:GD}, we investigate the implications of our analysis for the boosted DCA in the special case of gradient descent. In particular, if the objective satisfies a quadratic \L{}ojasiewicz condition (also often referred to as Polyak-\L{}ojasiewicz), we present results on the rate of linear convergence. Section~\ref{sec:conclusion} concludes with open questions and a further discussion of related literature.

\section{Performance estimation}
\label{sec:PEP}

A key step to the desired convergence results consists in formulating an abstract optimization problem that aims to find the worst-case input $ f_1 \in \mathcal{F}_{\mu_1,L_1}(\mathbb{R}^n)$ and
$f_2 \in \mathcal{F}_{\mu_2,L_2}(\mathbb{R}^n)$ for Algorithm \ref{Alg:BCDA}, where we view $\mu_1 \ge 0,\mu_2\ge 0,L_1 \in [\mu_1,\infty],L_2 \in [\mu_2,\infty]$ as given parameters:
\begin{equation}
\begin{aligned}
   \max & \ \left(\min_{1\leq k\leq N+1} \left\|g_1^k-g_2^k\right\|^2\right)\\  & \ f_1\in \mathcal{F}_{\mu_1,L_1}(\mathbb{R}^n), f_2\in \mathcal{F}_{\mu_2,L_2}(\mathbb{R}^n)\\
& f_1(x)-f_2(x)\geq f^\star  \ \ \ \forall x\in\mathbb{R}^n\\
  &  f_1(x^1)-f_2(x^1)-f^\star\leq \Delta\\
&  \  g_1^{N+1}, g_2^{N+1},  x^{N+1}, \ldots, x^2 \ \textrm{are generated  by Algorithm \ref{Alg:BCDA}  w.r.t.}\ f_1, f_2, x^1  \\
  & \ x^1\in\mathbb{R}^n,
\end{aligned}\label{eq:base_PEP}
\end{equation}
where $f_1, f_2$ and $x^k, g_1^k, g_2^k$ ($k\in\{1, ..., N+1\}$) are decision variables, and $\Delta,\mu_1,L_1,\mu_2,L_2, N$ are fixed parameters.

At first glance, this problem seems intractable, since we need to optimize over the infinite-dimensional function spaces
$\mathcal{F}_{\mu_1,L_1}(\mathbb{R}^n)$ and $\mathcal{F}_{\mu_2,L_2}(\mathbb{R}^n)$. However, since we only reference $f_1$ and $f_2$ (and their subgradients) at the $N$ iterates, we may reformulate these membership conditions as tractable interpolation conditions as in~\cite{taylor2017smooth}.

To this end, consider a finite index set $I$ (that will correspond to iterates in our analysis), and given a set of triplets $\left\{(x^k,g^k,f^k) \right\}_{k\in I}$ where  $x^k \in \mathbb{R}^n$,
 $g^k\in \mathbb{R}^n$ and   $f^k\in \mathbb{R}$.
 We are concerned with the following interpolation/extension result: for given $\mu \ge 0$ and $L \ge \mu$, does there exist a $f\in\mathcal{F}_{\mu, L}(\mathbb{R}^n)$ such that
$f(x^k)=f^k$, and $g^k\in\partial f(x^k)$ for all $k\in I$. If yes, we say that $\left\{(x^k,g^k,f^k) \right\}_{k\in I}$  is $\mathcal{F}_{\mu, L}(\mathbb{R}^n)$-interpolable.

\begin{theorem}(\cite{taylor2017smooth}, Theorem 4)
Assume $\mu \ge 0$ and $L \in [\mu,\infty]$.
The following statements are equivalent:
\begin{enumerate}
\item
$\left\{(x^i,g^i,f^i) \right\}_{i\in I}$  is $\mathcal{F}_{\mu, L}(\mathbb{R}^n)$-interpolable;
\item
 $\forall i,j\in I$:
\begin{equation}
\label{fmuLinterpolation}
\frac{1}{2(1-\frac{\mu}{L})}\left(\frac{1}{L}\left\|g^i-g^j\right\|^2+\mu\left\|x^i-x^j\right\|^2-\frac{2\mu}{L}\left\langle g^j-g^i,x^j-x^i\right\rangle\right)\leq
f^i-f^j-\left\langle g^j, x^i-x^j\right\rangle.
\end{equation}
\end{enumerate}
\end{theorem}
Using this theorem, as well as the following lemma, we may reformulate~\eqref{eq:base_PEP} as an~SDP.

\begin{lemma}[\cite{abbaszad2024DCA}, Lemma 2.1 (\emph{Descent lemma})]\label{Lem_de}
 Let $f_1\in\mathcal{F}_{\mu_1, L_1}({\mathbb{R}^n})$ and $f_2\in\mathcal{F}_{\mu_2, L_2}({\mathbb{R}^n})$ and let $f=f_1-f_2$. If $g_1\in\partial f_1(x)$ and $g_2\in\partial f_2(x)$, then
 $$
 f^\star\leq f(x)-\tfrac{1}{2\left(L_1-\mu_2\right)}\|g_1-g_2\|^2.
 $$
\end{lemma}

To reformulate the performance estimation problem, we need some more notation to accommodate the two different types of points $x^k$ and $y^k$ generated by Algorithm \ref{Alg:BCDA}.
In particular, we introduce the vector variable $g_1(x^k)$ to correspond to a subgradient of $f_1$ at $x^k$, $g_2(y^k)$ that corresponds to a subgradient of $f_2$ at $y^k$, etc. Similarly, we introduce scalar variables $f_1(x^k)$ to correspond to the function value of the function $f_1$ at $x^k$, etc. Finally, for easy reference, we define a set of variables:
\[
S := \left\{x^1,\ldots,x^{N+1},y^1,\ldots,y^{N}\right\}.
\]

We thus arrive at the following performance estimation problem:
\begin{equation}\label{PEP}
\begin{aligned}
  \max & \ \left(\min_{1\leq k\leq N+1} \left\|g_1(x^k)-g_2(x^k)\right\|^2\right)\\
 \mbox{s.t. } \ & \tfrac{1}{2(1-\tfrac{\mu_\ell}{L_\ell})}\left(\tfrac{1}{L_\ell}\left\|g_{\ell}(u)-g_{\ell}(v)\right\|^2+\mu_\ell\left\|u-v\right\|^2
-\tfrac{2\mu_\ell}{L_\ell}\left\langle g_{\ell}(v)-g_{\ell}(u),v-u\right\rangle\right)\leq\\
  & \ \ \ \ \ f_{\ell}(u)-f_{\ell}(v)-\left\langle g_{\ell}(v), u-v\right\rangle \ \ \forall  u,v \in S, \; \ell \in \{1,2\}\\
 &    g_{1}(y^k)=g_{2}(x^k) \ \ k\in\left\{1, \ldots, N\right\} \\
 & x^{k+1}= y^k + \alpha(y^{k}-x^{k}) \ \ k\in\left\{1, \ldots, N\right\} \\
&  f_1(u)-f_2(u)-\frac{1}{2(L_1-\mu_2)}\|g_1(u)-g_2(u)\|^2\geq f^\star  \ \ \forall  u \in S\\
  &  f_1(x^1)-f_2(x^1)-f^\star \leq \Delta,
\end{aligned}
\end{equation}
where $S, \; g_{\ell}(u), f_\ell(u) \; (\ell \in {1,2}, u \in S)$ are decision variables, and
where $\Delta,\mu_1,L_1,\mu_2,L_2, N$ are fixed parameters.

By the \emph{descent lemma} (Lemma~\ref{Lem_de}),
 the constraints $f(x)\geq f^\star$ for each $x\in\mathbb{R}^n$ is replaced by
\begin{equation}
\label{PEP_descent_lemma_constraints}
f_1(u)-f_2(u)-\frac{1}{2(L_1-\mu_2)}\|g_1(u)-g_2(u)\|^2\geq f^\star  \ \ \forall \; u \in S.
\end{equation}
The optimality conditions for the  convex subproblem,
\[
y^{k}\in \text{argmin}_{x\in\mathbb{R}^n} f_1(x)-\left(f_2(x^k)+\langle g_2^k, x-x^k\rangle\right),
\]
are equivalent to  $ g_{1}(y^k)=g_{2}(x^k)$ for some $g_1(y^k)\in\partial{f_1}(y^{k})$; see Theorem 3.63 in \cite{Amir} (Fermat's theorem). {(For a more detailed discussion of criticality in D.C. programming, see e.g.\ \cite[Section 2]{abbaszad2024DCA}.)}

The final problem \eqref{PEP} may be written as an SDP problem, since it is linear in the gram matrix of the vector variables, after using the equality
constraints to reduce variables.

To illustrate this fact, we provide some graphs where the SDPs were solved numerically through MOSEK~\cite{mosek}.
In Figure 1 we consider the worst-case scenario after $N=12$ iterations of Algorithm \ref{Alg:BCDA}, for growing values of $\alpha$, and for
$L_1=L_2=2$,  $\mu_1=1$,  and $\mu_2=0$. {Note that this choice of curvature parameters is not covered by the result in Theorem \ref{thm:main}, since there it is assumed that $\mu_1 = \mu_2 > 0$.}
In the figure we see that some choices of $\alpha$ result in poorer rates of convergence than $\alpha = 0$ (the usual DCA).
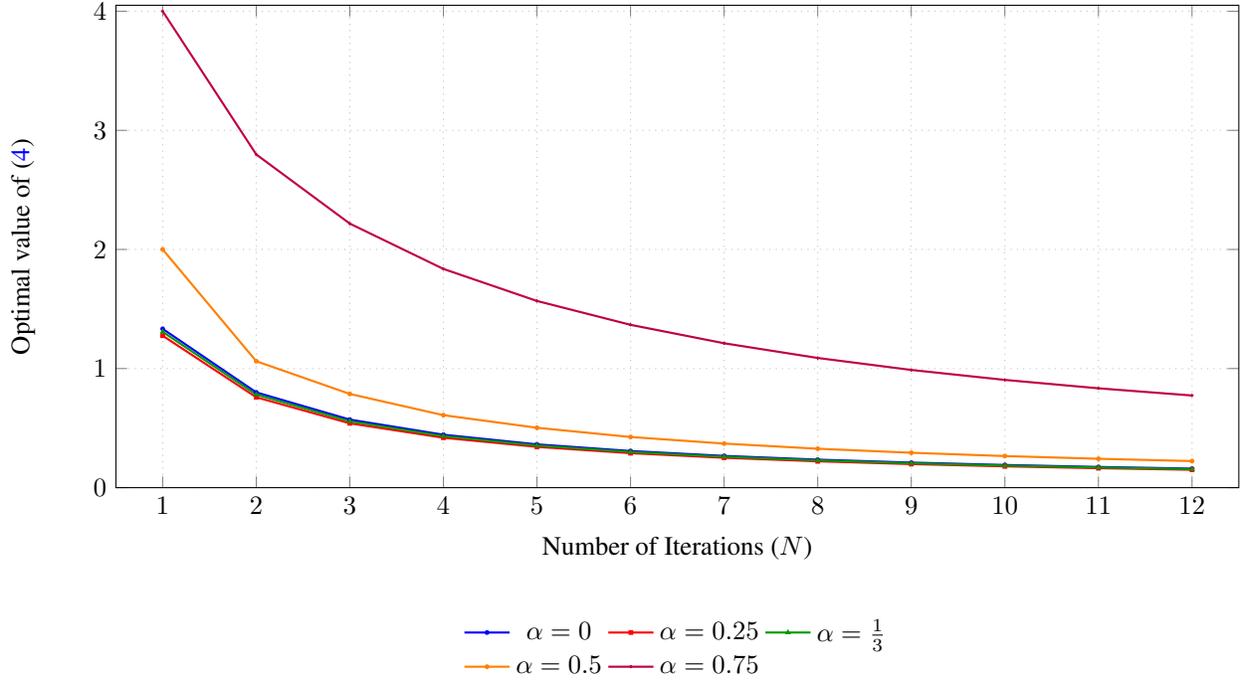
\begin{figure}[h!]
\label{fig:1}
\centering
\begin{tikzpicture}
\begin{axis}[
    width=1\textwidth,
    height=8cm,
    ylabel={Optimal value of \eqref{PEP}},
    xlabel={Number of Iterations ($N$)},
    xmin=0.5, xmax=12.5,
    ymin=0, ymax=4.05,
    xtick={1,...,12},
    yticklabel style={/pgf/number format/fixed, /pgf/number format/precision=2},
    grid=both,
    major grid style={dotted,gray!50},
    minor grid style={dashed,gray!20},
    minor x tick num=0,
    minor y tick num=0,
    legend style={
        at={(0.5,-0.25)}, % Below the plot
        anchor=north,
        legend columns=3,
        draw=none,
        fill=none,
        cells={align=left}
    }
]

% === Plot each alpha's CSV data with legend ===
\addplot[
    thick, color=blue, mark=*, mark options={solid}, mark size=.5
] table[x=N,y=OPT_Value,col sep=comma] {Data/M2_0/OPT_alpha000.csv};
\addlegendentry{$\alpha = 0$}

\addplot[
    thick, color=red, mark=square*, mark options={solid}, mark size=.5
] table[x=N,y=OPT_Value,col sep=comma] {Data/M2_0/OPT_alpha025.csv};
\addlegendentry{$\alpha = 0.25$}

\addplot[
    thick, color=green!60!black, mark=triangle*, mark options={solid}, mark size=.5
] table[x=N,y=OPT_Value,col sep=comma] {Data/M2_0/OPT_alpha33.csv};
\addlegendentry{$\alpha = \frac{1}{3}$}

\addplot[
    thick, color=orange, mark=o, mark options={solid}, mark size=.5
] table[x=N,y=OPT_Value,col sep=comma] {Data/M2_0/OPT_alpha050.csv};
\addlegendentry{$\alpha = 0.5$}

\addplot[
    thick, color=purple, mark=x, mark options={solid}, mark size=.5
] table[x=N,y=OPT_Value,col sep=comma] {Data/M2_0/OPT_alpha075.csv};
\addlegendentry{$\alpha = 0.75$}

\end{axis}
\end{tikzpicture}
\caption{Performance bounds for DCA ($\alpha=0$) and boosted DCA when $L_1=L_2=2$, and $\mu_1=1,\ \mu_2=0$.}
\end{figure}

In Figure \ref{fig:2}, we again consider $N=12$ iterations, but this time with the curvature parameters $L_1=L_2=2$, and $\mu_1=0,\ \mu_2=1$.
Here, one does see improving rates of convergence for growing values of $\alpha$, i.e.\ boosted DCA does outperform DCA here,
{although this is not predicted by Theorem \ref{thm:main}, since $\mu_1=0$.
}
\begin{figure}[h!]
\centering
\begin{tikzpicture}
\begin{axis}[
    width=1\textwidth,
    height=8cm,
    ylabel={Optimal value of \eqref{PEP}},
    xlabel={Number of Iterations ($N$)},
    xmin=0.5, xmax=12.5,
    ymin=0, ymax=1.02,
    xtick={1,...,12},
    yticklabel style={/pgf/number format/fixed, /pgf/number format/precision=2},
    grid=both,
    major grid style={dotted,gray!50},
    minor grid style={dashed,gray!20},
    minor x tick num=0,
    minor y tick num=1,
    legend style={
        at={(0.5,-0.25)}, % Below the plot
        anchor=north,
        legend columns=3,
        draw=none,
        fill=none,
        cells={align=left}
    }
]

% === Plot each alpha's CSV data with legend ===
\addplot[
    thick, color=blue, mark=*, mark options={solid}, mark size=.5
] table[x=N,y=OPT_Value,col sep=comma] {Data/M1_0/OPT_alpha000.csv};
\addlegendentry{$\alpha = 0$}

\addplot[
    thick, color=red, mark=square*, mark options={solid}, mark size=.5
] table[x=N,y=OPT_Value,col sep=comma] {Data/M1_0/OPT_alpha025.csv};
\addlegendentry{$\alpha = 0.25$}

\addplot[
    thick, color=green!60!black, mark=triangle*, mark options={solid}, mark size=.5
] table[x=N,y=OPT_Value,col sep=comma] {Data/M1_0/OPT_alpha33.csv};
\addlegendentry{$\alpha = \frac{1}{3}$}

\addplot[
    thick, color=orange, mark=o, mark options={solid}, mark size=.5
] table[x=N,y=OPT_Value,col sep=comma] {Data/M1_0/OPT_alpha050.csv};
\addlegendentry{$\alpha = 0.5$}

\addplot[
    thick, color=purple, mark=x, mark options={solid}, mark size=.5
] table[x=N,y=OPT_Value,col sep=comma] {Data/M1_0/OPT_alpha075.csv};
\addlegendentry{$\alpha = 0.75$}

\end{axis}
\end{tikzpicture}
\caption{\label{fig:2} Performance bounds for DCA ($\alpha=0$) and boosted DCA when $L_1=L_2=2$, and $\mu_1=0,\ \mu_2=1$.}
\end{figure}
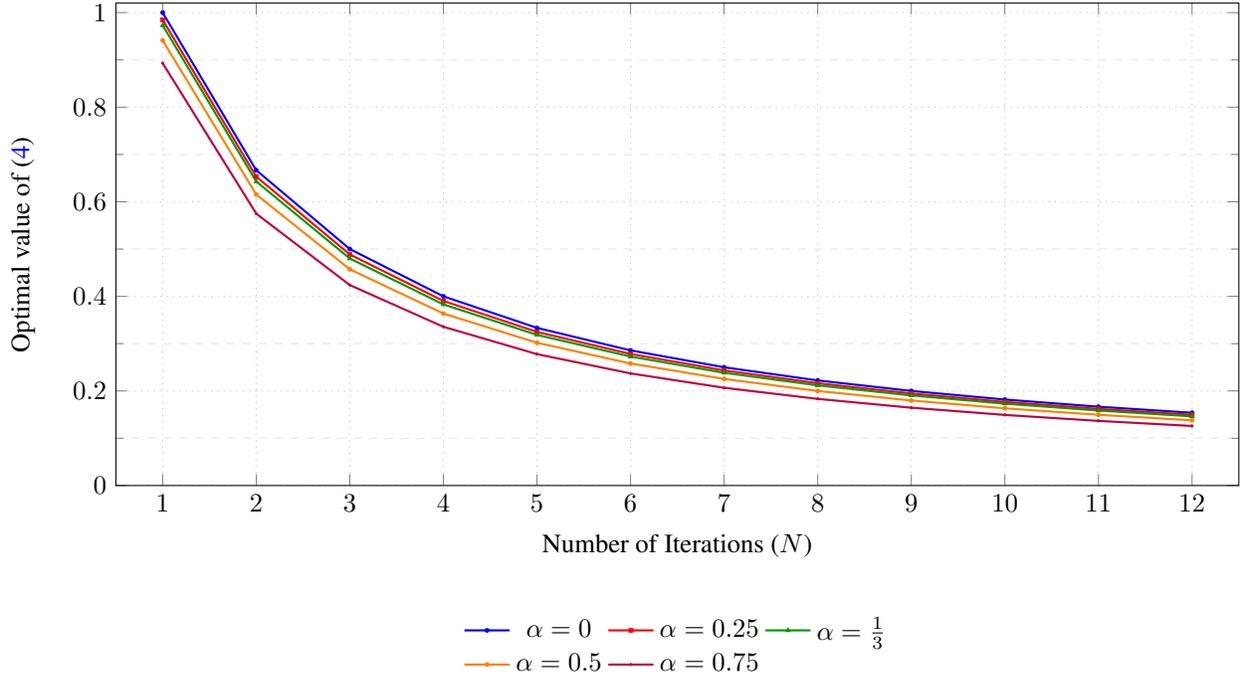

In Figure \ref{fig:3}, we plot how the convergence criterion depends on $\alpha$, for $N=10$ iterations and
parameter values $L_1=1,\ L_2=2$, and $\mu_1=0.5,\ \mu_2=0.25$. Note that the graph is a convex function of $\alpha$ with minimizer $\alpha^\star \approx 0.4$. Also note
that --- when the value of $\alpha$ gets too large --- the worst-case performance of DCA is actually better.
{Once again, the results in Figure \ref{fig:3} are not covered by Theorem \ref{thm:main}, since $\mu_1\neq \mu_2$ here.
}
\begin{figure}[h!]
\centering
\begin{tikzpicture}
\begin{axis}[
    width=1\textwidth,
    height=7cm,
    xlabel={$\alpha$},
    ylabel={Optimal value of \eqref{PEP}},
    xmin=-0.05, xmax=1.05,
    grid=both,
    major grid style={dotted,gray!50},
    minor grid style={dashed,gray!20},
    minor x tick num=0,
    minor y tick num=0
]

% Plot data from CSV
\addplot[
    thick,
    color=blue,
    mark=*,
    mark options={solid},
    mark size=.5
]
table[x=alpha, y=OPT_Value, col sep=comma] {Data/OPT_vs_alpha.csv};

\end{axis}
\end{tikzpicture}
\caption{\label{fig:3} Performance bounds for DCA ($\alpha=0$) and boosted DCA for the case that $L_1=1\ L_2=2$, and $\mu_1=0.5,\ \mu_2=0.25$ and $N=10$.}
\end{figure}
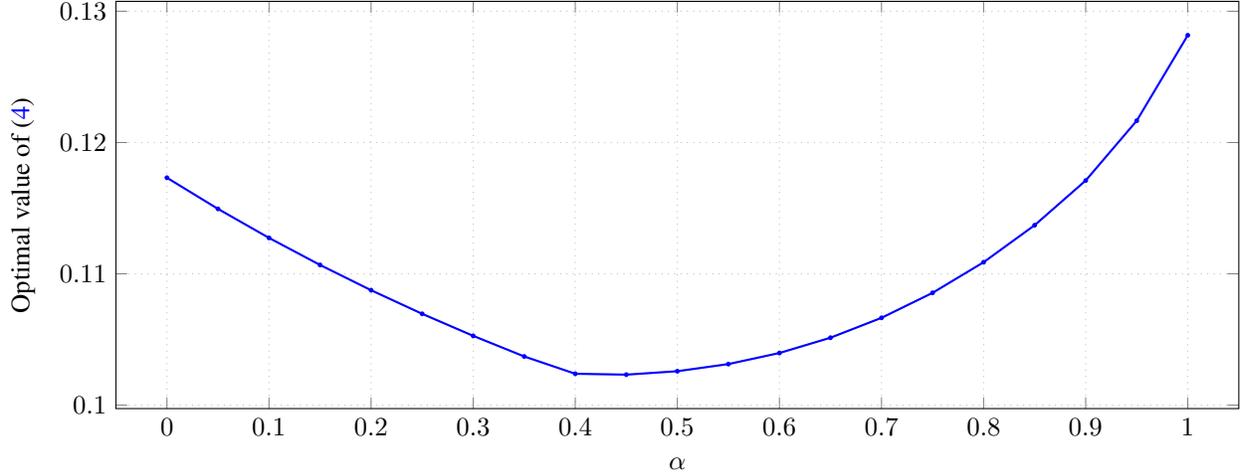

\subsection*{Adding the Polyak-\L ojasiewicz (P{\L}) inequality}
We may also consider the DC problem \eqref{P} in the case where the objective satisfies the Polyak-\L ojasiewicz (P{\L}) inequality,
defined as follows.

\begin{dfn}
A differentiable function $f$ is said to satisfy the P{\L}\ inequality on $X\subseteq \mathbb{R}^n$ if there exists $\eta>0$ such that
\[
f(x)-f^\star \leq \tfrac{1}{2\eta} \| \nabla f(x)\|^2, \ \ \ \forall x\in X.
\]
\end{dfn}
Under the P{\L}\ inequality assumption, every stationary point of $f$ is a global minimizer, but
it
is  a weaker assumption  than convexity.

As done in the performance analysis of DCA in \cite{abbaszad2024DCA}, one may
incorporate  the P{\L}\ inequality in the SDP problem \eqref{PEP}, by adding the constraints:
\begin{equation}
\label{PL constraint}
f_1(u) - f_2(u) -f^\star \leq \tfrac{1}{2\eta} \| g_1(u) - g_2(u)\|^2 \; \forall u \in S.
\end{equation}
Here we implicitly assume that the variables in $S$ only take values in $X$.

Note that one could, instead of using~\eqref{PL constraint}, build on~\cite[Proposition 3.4]{rubbens2025constructive} to obtain a more accurate picture for this class of functions (at the cost of much more complicated formulations). This is done in the numerical experiments below.

\section{Sublinear convergence of boosted DCA}
\label{sec:pep results}
In this section we give a proof of Theorem \ref{thm:main}.
The proof proceeds in the usual manner for SDP performance estimation: we provide suitable multipliers for the constraints
of problem \eqref{PEP}, and then aggregate the constraints with these multipliers to obtained the required inequality.
The values of the multipliers in Table \ref{tab:1} below were obtained in a computer-assisted manner: after solving the (dual) SDP for different choices of the parameters,
we could guess the analytic expressions for the multipliers, and subsequently verify the guesses in a rigorous manner. In other words, our proof of
Theorem \ref{thm:main} was `discovered' in a computer-assisted manner, but the proof itself is self-contained, and does not rely on any computational results.

\begin{proof} (of Theorem \ref{thm:main})
We decompose the proof into three parts:
\begin{enumerate}
    \item Reformulate the desired statement.
    \item Show that the reformulation can be proved using a simple one-iteration analysis and using induction.
    \item Perform the one-iteration analysis.
\end{enumerate}

\paragraph{Part 1 of the proof: reformulation.}
The main result follows if we prove the following inequality:

\begin{equation}\label{eq:target}
    \begin{split}
       & \left[\frac12+\alpha\frac{\mu}{L}\right]\|g_1(x^1)-g_2(x^1)\|^2 +
        \sum_{i=2}^N\left[1+\alpha\frac{\mu}{L}\right]\|g_1(x^i)-g_2(x^i)\|^2+\left[\frac{1-\frac12\frac\mu L}{1-\frac\mu  L}\right]\|g_1(x^{N+1})-g_2(x^{N+1})\|^2\\
       \leq &  L \left[ f_1(x^1)-f_2(x^1)-f_1(x^{N+1})+f_2(x^{N+1})+\frac{1}{2(L-\mu)}\|g_1(x^{N+1})-g_2(x^{N+1})\|^2 \right].
    \end{split}
\end{equation}
Indeed, using the upper bound on $f^\star$ from Lemma \ref{Lem_de} with $x = x^{N+1}$, we get
\begin{equation*}
 \begin{split}
   & \left[\frac12+\alpha\frac{\mu}{L}\right]\|g_1(x^1)-g_2(x^1)\|^2 +
     \sum_{i=2}^N\left[1+\alpha\frac{\mu}{L}\right]\|g_1(x^i)-g_2(x^i)\|^2 +
    \left[\frac{1-\frac12\frac\mu L}{1-\frac\mu L}\right]\|g_1(x^{N+1})-g_2(x^{N+1})\|^2 \\
   \le & L \left[ f_1(x^1)-f_2(x^1)-f^\star \right].
   \end{split}
\end{equation*}
    Then, lower-bounding the left-hand side by the smallest gradient norm observed, we obtain:
\begin{equation*}
    \begin{aligned}
        \min_{1\leq i\leq N+1}\left(\|g_1(x^i)-g_2(x^i)\|^2\right)\left[\frac12+\alpha\frac{\mu}{L}+\sum_{i=2}^N\left(1+\alpha\frac{\mu}{L}\right)+\frac{1-\frac12\frac\mu L}{1-\frac\mu L}\right]\leq L \left[ f_1(x^1)-f_2(x^1)-f^\star \right],
    \end{aligned}
\end{equation*}
where $\left[\frac12+\alpha\frac{\mu}{L}+\sum_{i=2}^N\left(1+\alpha\frac{\mu}{L}\right)+\frac{1-\frac12\frac\mu L}{1-\frac\mu L}\right]=(1+\kappa\alpha)N + \frac{1}{2(1-\kappa)}$, yielding the desired result.

\paragraph{Part 2 of the proof: reducing to single iteration analysis.}
We may prove that the required result \eqref{eq:target} holds by proving that the following inequality holds for a given iteration $k \in \mathbb{N}$:

\begin{equation}\label{eq:one_iteration_inequality}
           f(x^{k+1})+\frac{1}{2L}\|g_1(x^{k+1})-g_2(x^{k+1})\|^2+
         \frac1L \left[\frac12 +\alpha \frac\mu L\right]\|g_1(x^k)-g_2(x^k)\|^2
         \leq  f(x^k).
  \end{equation}
We will verify in part 3 of this proof that \eqref{eq:one_iteration_inequality} indeed holds at each iteration, but first show that it implies  the required result \eqref{eq:target}.
Assuming now that \eqref{eq:one_iteration_inequality} holds for any $k \in \mathbb{N}$, we will prove by induction that, for any $k \ge 2$, one has
\begin{equation}
\label{eq:induction}
    \begin{split}
        & f(x^k)+\frac1L\left[\frac12+\alpha\frac{\mu}{L}\right]\|g_1(x^1)-g_2(x^1)\|^2+
         \sum_{i=2}^{k-1}\frac1L\left[1+\alpha\frac{\mu}{L}\right]\|g_1(x^i)-g_2(x^i)\|^2+\frac{1}{2L}\|g_1(x^k)-g_2(x^k)\|^2 \\
        \leq & f(x^1),
    \end{split}
\end{equation}
with the usual convention that the empty sum is zero (for $k=2$).

The base case for the induction argument is obtained by setting $k=1$ in \eqref{eq:one_iteration_inequality}.
For the induction step, let $k\geq 2$ and assume that \eqref{eq:induction} holds.
Then applying inequality~\eqref{eq:one_iteration_inequality} to lower bound $f(x^k)$ yields
\begin{equation*}
    \begin{split}
        & f(x^{k+1})+\frac1L\left[\frac12+\alpha\frac{\mu}{L}\right]\|g_1(x^1)-g_2(x^1)\|^2+\sum_{i=2}^{k}\frac1L\left[1+\alpha\frac{\mu}{L}\right]\|g_1(x^i)-g_2(x^i)\|^2+\frac{1}{2L}\|g_1(x^{k+1})-g_2(x^{k+1})\|^2 \\
        \leq & f(x^1).
    \end{split}
\end{equation*}
Thus we have shown that \eqref{eq:induction} holds for all $k \ge 2$.
For $k=N$, we therefore obtain
\begin{equation}\label{eq:final_recursion}
    \begin{split}
      &  \frac1L\left[\frac12+\alpha\frac{\mu}{L}\right]\|g_1(x^1)-g_2(x^1)\|^2+\sum_{i=2}^{N}\frac1L\left[1+\alpha\frac{\mu}{L}\right]\|g_1(x^i)-g_2(x^i)\|^2+\frac{1}{2L}\|g_1(x^{N+1})-g_2(x^{N+1})\|^2 \\
      \leq  & f(x^1)-f(x^{N+1}).
    \end{split}
\end{equation}
Proving~\eqref{eq:target} requires lower-bounding \[ f(x^1)-f(x^{N+1})+\frac{1}{2(L-\mu)}\|g_1(x^{N+1})-g_2(x^{N+1})\|^2,\] as opposed to $f(x^1)-f(x^{N+1})$.
To adjust the bound~\eqref{eq:final_recursion} to get to the desired inequality, we simply add $\frac{1}{2(L-\mu)}\|g_1(x^{N+1})-g_2(x^{N+1})\|^2$ on both sides; this leads directly to~\eqref{eq:target} (after scaling by $L$).
%Note that the coefficient $\left[\frac{1-\frac12\frac\mu L}{1-\frac\mu L}\right]$ in front
% of the last gradient term $\|g_1(x^{N+1})-g_2(x^{N+1})\|^2$ (in the left-hand side of the inequality)
% evaluates to $\left[\frac{1-\frac12\frac\mu L}{1-\frac\mu L}\right]=\frac{1}{2}+\frac{L}{2(L-\mu)}$.

\paragraph{Part 3 of the proof: performing the single iteration analysis.}
We now verify that \eqref{eq:one_iteration_inequality} indeed holds at each iteration $k$.

Recall that  $f_1,f_2 \in \mathcal{F}_{\mu,L}(\mathbb{R}^n)$, and therefore, by \eqref{fmuLinterpolation}, satisfy
\begin{equation}
Q_\ell(u,v) \le 0, \quad \mbox{($u,v \in \mathbb{R}^n$,   $\ell \in \{1,2\}$)},
\label{eq:interpolation constraints}
\end{equation}
where we define
\begin{eqnarray*}
Q_\ell(u,v) &=& \tfrac{1}{2(1-\tfrac{\mu}{L})}\left(\tfrac{1}{L}\left\|g_{\ell}(u)-g_{\ell}(v)\right\|^2+\mu\left\|u-v\right\|^2
 -\tfrac{2\mu}{L}\left\langle g_{\ell}(v)-g_{\ell}(u),v-u\right\rangle\right)  \\
&& - f_{\ell}(u)+f_{\ell}(v)+\left\langle g_{\ell}(v), u-v\right\rangle, \\
\end{eqnarray*}
and where $g_\ell(u)$ denotes a subgradient of $f_\ell$ at $u$, as before.

We now consider the inequalities \eqref{eq:interpolation constraints} for specific choices of $(u,v)$ as listed in Table \ref{tab:1},
together with corresponding multipliers, $\lambda_{\ell,u,v}$, also listed in the table.
{\renewcommand{\arraystretch}{1.3}\begin{table}[h!]
\begin{center}
\begin{tabular}{@{}llllll@{}}
\toprule
  $(u,v)$    & $(x^k,y^k)$& $(y^k,x^k)$ & $(x^{k+1},y^k)$& $(y^k,x^{k+1})$& $(x^{k+1},x^k)$\\  \midrule
  $\ell$     &  $1$ & $1$ & $1$ & $1$ & $2$ \\
  $\lambda_{\ell,u,v}$ & $1 + \frac{2\alpha \mu}{L}$ & $\frac{2\alpha \mu}{L} $ & $\frac{1}{\alpha} - 1$  & $\frac{1}{\alpha}$ & $1$\\
  \bottomrule
\end{tabular}
\end{center}
\caption{\label{tab:1} Multipliers for the interpolation constraints $Q_\ell(u,v) \le 0$ for a fixed index $k$ and specific choices for $(u,v)$.}
\end{table}}
Note that all the multipliers in the table are nonnegative if $\alpha \in (0,1]$.
Multiplying each inequality from \eqref{eq:interpolation constraints} by the corresponding nonnegative multiplier from Table \ref{tab:1} for the specific choice
of $(u,v)$ listed in the table, and summing, we obtain the inequality:
\begin{equation*}
    \begin{aligned}
        0\geq & f_1(x^{k+1})-f_2(x^{k+1})-f_1(x^k)+f_2(x^k)+\tfrac1{2L} \|g_1(x^{k+1})-g_2(x^{k+1})\|^2+\tfrac1L \left(\frac12+\alpha\frac\mu L\right)\|g_1(x^k)-g_2(x^k)\|^2\\
        &+\tfrac{1}{2 \alpha  L (L-\mu ) (\alpha  \mu +2 (1-\alpha ) L)} \\
        &\quad\times\| (\alpha  \mu -2 (\alpha -1) L)g_1(x^{k+1})+L(\alpha-2)g_2(x^k)+\alpha(L-\mu)g_2(x^{k+1})+\alpha  L (-\alpha  \mu +\mu +L)(x^k-y^k)\|^2\\
        &+\tfrac{\alpha L-2\mu}{2(L-\mu)(2L(\alpha-1)-\alpha \mu)}\|-g_2(x^k)+g_2(x^{k+1})+(L+\alpha \mu)(x^k-y^k)\|^2\\
        &+\tfrac{\mu(L+2\alpha L+2\alpha \mu)}{2L^2(L-\mu)}\|g_1(x^k)-g_2(x^k)-L(x^k-y^k)\|^2.
    \end{aligned}
\end{equation*}
The coefficients of the last three terms are nonnegative when $0<\alpha \leq \min \left\{2\frac{\mu}{L}, \frac1{1-\frac{1}{2}\frac{\mu}{L}}\right\}=2\kappa$. Removing the last three nonnegative terms on the right-hand-side, one arrives to the desired
\begin{equation*}
    \begin{aligned}
         0 \geq &f_1(x^{k+1})-f_2(x^{k+1})-f_1(x^k)+f_2(x^k)+\tfrac1{2L} \|g_1(x^{k+1})-g_2(x^{k+1})\|^2+\tfrac1L \left(\frac12+\alpha\frac\mu L\right)\|g_1(x^k)-g_2(x^k)\|^2\\
         =&f(x^{k+1})-f(x^k)+\tfrac1{2L} \|g_1(x^{k+1})-g_2(x^{k+1})\|^2+\tfrac1L \left(\frac12+\alpha\frac\mu L\right)\|g_1(x^k)-g_2(x^k)\|^2,
    \end{aligned}
\end{equation*}
and the proof therefore holds when $0<\alpha\leq \min\left\{2\kappa,1\right\}$.
\end{proof}

{
\section{Sublinear convergence of boosted DCA with backtracking}
\label{sec:backtracking}
By Theorem \ref{thm:main}, the optimal choice for fixed $\alpha > 0$ in Algorithm \ref{Alg:BCDA} is $\alpha = \min \{1,2\kappa\}$. However, this requires knowledge of the parameter~$\kappa = \mu/L$, while the usual DCA (with $\alpha = 0$) does not require knowledge of the parameters~$\mu$ and~$L$.

By leveraging our proof technique for Theorem \ref{thm:main}, we may also analyse a variant of boosted CDA where the parameters $\mu$ and $L$ are not known, but only estimated at each step, and where we use backtracking to find a suitable step length.

Backtracking analysis dates back to the work by Armijo~\cite{Armijo} and Goldstein~\cite{goldstein1962cauchy} (see, e.g.,\cite[Corollary 2]{Armijo}), and our aim here is only to show that this type of (standard) analysis may be combined with our proof technique to obtain a parameter-free variant of Algorithm~\ref{Alg:BCDA}.

With reference to Algorithm \ref{alg:backtracking}, the main ideas of the backtracking analysis are as follows:
\begin{itemize}
    \item 
At each iteration $k$, we have estimates $\mu^{(k)}$ of $\mu$ and 
$L^{(k)}$ of $L$ respectively.
    \item
    By the proof of Theorem \ref{thm:main}, the following
    inequality  holds at each iteration $k$ if using the correct values of~$(\mu,L,\alpha)$:\begin{equation}\label{eq:one_iteration_inequality}
           f(x^{k+1})+\frac{1}{2L}\|g_1(x^{k+1})-g_2(x^{k+1})\|^2+
         \frac1L \left[\frac12 +\alpha \frac\mu L\right]\|g_1(x^k)-g_2(x^k)\|^2
         \leq  f(x^k).
  \end{equation}
  At each iteration, we test if it also holds for the estimates $\mu^{(k)}$,
$L^{(k)}$ and $\alpha^{(k)} := \min \{1,2\mu^{(k)}/L^{(k)}\}$. If not, we adjust these estimates as detailed in step 6 of Algorithm \ref{alg:backtracking}. In particular, we use a fixed parameter $\beta > 1$ to decrease $\mu^{(k)}$
via $\mu^{(k)}\leftarrow \mu^{(k)}/\beta$, and increase $L^{(k)}$ via $L^{(k)}\leftarrow \beta L^{(k)}$, if needed.
\item 
In the statement of Theorem~\ref{thm:backtracking} below, the quantity $\bar{K}$ bounds the maximum number of parameter updates in step 6 of Algorithm~\ref{alg:backtracking}. Thus, for each iteration $k$, one has $L^{(k)}\leq \bar{L}:=\beta^{\bar{K}}\max\{L^{(0)}, L\}$ and $\mu^{(k)}\geq\bar{\mu}:= \beta^{-\bar{K}}\min\{\mu^{(0)},\mu\}$. Thus, we also have $\alpha^{(k)}\geq \bar{\alpha}:=\min\left\{2\frac{\bar{\mu}}{\bar{L}},1\right\}$ and $\kappa\geq \bar{\kappa}:=\frac{\bar{\mu}}{\bar{L}}$.
\end{itemize}

\begin{algorithm}[H]
\caption{Boosted DCA with backtracking\label{alg:backtracking}}
\begin{algorithmic}
\STATE Pick $x^1\in\mathbb{R}^n$ (starting point), $N \in \mathbb{N}$ (number of iterations), ${\mu}^{(0)}$ (estimate of strong convexity), ${L}^{(0)}$ (estimate of smoothness), and $\beta>1$ (backtracking strength) with $0<{\mu}^{(0)}<{L}^{(0)}$.
\STATE For $k=1, 2, \ldots, N$ perform the following steps:\\
\begin{enumerate}
\item Pick   $g_2^{k}\in\partial f_2(x^{k})$.
\item
Set
\begin{align*}
  y^{k}\in \text{argmin}_{x\in\mathbb{R}^n} f_1(x)-\left(f_2(x^k)+\langle g_2^k, x-x^k\rangle\right).
\end{align*}
\item
Set $d^k = y^k - x^k$
\item Set $L^{(k)}\leftarrow L^{(k-1)}$, $\mu^{(k)}\leftarrow \mu^{(k-1)}$ 
\item Set $\alpha^{(k)} \leftarrow \min\left\{2\tfrac{{\mu^{(k)}}}{{L^{(k)}}},1\right\}$ and $x^{k+1} \leftarrow y^k + \alpha^{(k)} d^k$
\item While~\eqref{eq:one_iteration_inequality} is not satisfied with $\mu\leftarrow {\mu}^{(k)}$, $L\leftarrow L^{(k)}$, $\alpha\leftarrow \alpha^{(k)}$ do
\begin{align*}
  L^{(k)}&\leftarrow \beta L^{(k)} \text{ (increase estimate of $L$) }\\
  \mu^{(k)}&\leftarrow \mu^{(k)}/\beta \text{ (decrease estimate of $\mu$). }\\
  \alpha^{(k)} &\leftarrow \min\left\{2\tfrac{{\mu^{(k)}}}{{L^{(k)}}},1\right\} \text{ (decrease boosted step)}\\
  x^{k+1} &\leftarrow y^k + \alpha^{(k)} d^k \text{ (new trial boosted step)}
\end{align*}
\end{enumerate}
\end{algorithmic}
\end{algorithm}
The following theorem is a direct analog of Theorem \ref{thm:main}, but for the worst-case analysis of the backtracking  Algorithm~\ref{alg:backtracking}.

\begin{theorem}
\label{thm:backtracking}
Let $f_1\in\mathcal{F}_{\mu, L}({\mathbb{R}^n})$ and $f_2\in\mathcal{F}_{\mu, L}({\mathbb{R}^n})$, $0 \le \mu < L < \infty$. After $N$ iterations of boosted DCA with backtracking (Algorithm \ref{alg:backtracking}) with initial guesses $0<\mu^{(0)}<L^{(0)}$ and backtracking parameter $\beta>1$, one has,
 \[
    \min_{1\leq k\leq N+1}\frac{\left\|g_1^k-g_2^k\right\|^2}{\bar{L}}\leq \frac{(f(x^1)-f^\star) }{(1+\bar{\kappa}\bar{\alpha})N + \frac{1}{2(1-\bar{\kappa})}},
\]
where  $\bar{K}:= \max\left\{0,\lceil \log_{\beta}\left(\beta L/L^{(0)}\right)\rceil, \lceil \log_{\beta}\left(\beta\mu^{(0)}/\mu\right)\rceil\right\}$, $\bar{L}:=\beta^{\bar{K}}\max\{L^{(0)}, L\}$, $\bar{\mu}:= \beta^{-\bar{K}}\min\{\mu^{(0)},\mu\}$, $\bar{\alpha}:=\min\left\{2\frac{\bar{\mu}}{\bar{L}},1\right\}$, and $\bar{\kappa}:=\frac{\bar{\mu}}{\bar{L}}$.
\end{theorem}

\paragraph{Proof:} At each iteration $k$, it holds that
\begin{equation}\label{eq:one_iteration_bt}
           f(x^{k+1})+\frac{1}{2L^{(k)}}\|g_1(x^{k+1})-g_2(x^{k+1})\|^2+
         \frac1{L^{(k)}} \left[\frac12 +\alpha^{(k)} \frac{\mu^{(k)}}{L^{(k)}}\right]\|g_1(x^k)-g_2(x^k)\|^2
         \leq  f(x^k).
  \end{equation}
By definition of $\bar{L}\geq L^{(k)}$, $\bar{\mu}\leq \mu^{(k)}$, and $\bar{\alpha}\leq \alpha^{(k)}$, we get
\begin{equation}\label{eq:one_iteration_bt}
           f(x^{k+1})+\frac{1}{2\bar{L}}\|g_1(x^{k+1})-g_2(x^{k+1})\|^2+
         \frac1{\bar{L}} \left[\frac12 +\bar{\alpha} \frac{\bar{\mu}}{\bar{L}}\right]\|g_1(x^k)-g_2(x^k)\|^2
         \leq  f(x^k),
  \end{equation}
 which is satisfied at all iterations (given that after at most $\bar{K}$ updates of the Lipschitz and strong convexity estimates, the corresponding values are both valid smoothness and strong convexity parameters for the problem at hand).
 
By induction, we thereby also get:
\begin{equation*}
    \begin{split}
        & f(x^{k+1})+\frac1{\bar{L}}\left[\frac12+\bar{\alpha}\frac{\bar{\mu}}{\bar{L}}\right]\|g_1(x^1)-g_2(x^1)\|^2+\sum_{i=2}^{k}\frac1{\bar{L}}\left[1+\bar{\alpha}\frac{\bar{\mu}}{\bar{L}}\right]\|g_1(x^i)-g_2(x^i)\|^2+\frac{1}{2\bar{L}}\|g_1(x^{k+1})-g_2(x^{k+1})\|^2 \\
        \leq & f(x^1).
    \end{split}
\end{equation*}
For $k=N$, we therefore obtain
\begin{equation}\label{eq:final_recursion}
    \begin{split}
      &  \frac1{\bar{L}}\left[\frac12+\bar{\alpha}\frac{\bar{\mu}}{\bar{L}}\right]\|g_1(x^1)-g_2(x^1)\|^2+\sum_{i=2}^{N}\frac1{\bar{L}}\left[1+\bar{\alpha}\frac{\bar{\mu}}{\bar{L}}\right]\|g_1(x^i)-g_2(x^i)\|^2+\frac{1}{2\bar{L}}\|g_1(x^{N+1})-g_2(x^{N+1})\|^2 \\
      \leq  & f(x^1)-f(x^{N+1}).
    \end{split}
\end{equation}
To arrive at the target inequality, we have to lower bound $f(x^1)-f^\star$. Since the objective function is also $\bar{L}$-smooth and $\bar{\mu}$-strongly convex, we can use the trick from the  proof of Theorem \ref{thm:main}, and instead lower bound
\[ f(x^1)-f(x^{N+1})+\frac{1}{2(\bar{L}-\bar{\mu})}\|g_1(x^{N+1})-g_2(x^{N+1})\|^2.\]
To adjust the bound~\eqref{eq:final_recursion} to get to the desired inequality, we simply add $\frac{1}{2(\bar{L}-\bar{\mu})}\|g_1(x^{N+1})-g_2(x^{N+1})\|^2$ on both sides. \qed
}

\section{Consequences for gradient descent}
\label{sec:GD}
In this section we look at the implications of our approach in the special case of gradient descent under the P{\L} assumption.
As will become clear, our approach allows for an improved analysis for $L$-smooth functions that meet the P{\L} condition.
Of course, one may also study performance analysis of gradient descent for such functions directly, as was done in \cite{abbaszadehpeivasti2022PL}, but we will show improved results via the boosted DCA analysis.

Consider the following optimization problem
\begin{align*}
    \inf_{x\in\mathbb{R}^n} \ f(x)
\end{align*}
where $f$ is lower-bounded by $f^\star$ and is an $L$-smooth function on $\mathbb{R}^n$ with $L<\infty$, that is $f\in\mathcal{F}_{-L,L}(\mathbb{R}^n)$ and differentiable. It is easily seen that the function $f$ can be written as
\[
f(x):=\frac{L}{2}\|x\|^2-\left(\frac{L}{2}\|x\|^2-f(x)\right) =  f_1(x) - f_2(x),
\]
where we define $f_1(x):=\frac{L}{2}\|x\|^2$ and $f_2(x):=\frac{L}{2}\|x\|^2-f(x)$.
Both functions $f_1$ and $f_2$ are convex since $f$ is $L$-smooth. If we solve the subproblem of the DCA (Algorithm \ref{Alg1})
at iteration $k$, namely
\[
x^{k+1}=\arg\min_x\ \frac{L}{2}\|x\|^2-\left(\frac{L}{2}\left\|x^k\right\|^2-f(x^k)\right)-\left\langle {L}x^k-\nabla f(x^k), x-x^k\right\rangle,
\]
we get
\[
x^{k+1}=x^k-\frac{1}{L}\nabla f(x^k),
\]
which is exactly the gradient descent method with the classical step length $\frac{1}{L}$.
Similarly, the boosted DCA (Algorithm~\ref{Alg:BCDA}) corresponds to gradient descent with constant step lengths
$(1+\alpha)\frac{1}{L}$. This allows us to consider longer step lengths than $1/L$.

We may now adapt the SDP performance analysis problem \eqref{PEP} for the special case where $f_1(x) = \frac{L}{2}\|x\|^2$
  and $f_2 \in\mathcal{F}_{0, 2L}({\mathbb{R}^n})$, i.e.\ $\mu_2 = 0$ and $L_2 = 2L$.
  Note that we now have
  $
  g_1(x) = Lx,
  $
  so that several constraints in \eqref{PEP} simplify.
In particular, one may readily derive the following linear convergence result for boosted DCA under the P{\L} inequality.

\begin{theorem}
\label{thm:GD PL}
Let $f_1(x)=\frac{L}{2}\|x\|^2$, $f_2\in\mathcal{F}_{0, 2L}({\mathbb{R}^n})$ and $f(x)=f_1(x)-f_2(x)$. If
     $f$ satisfies P{\L} inequality on $X=\{x: f(x)\leq f(x^1)\}$ with modulus $0<\eta\leq L$, then for $x^1, x^2$  from boosted DCA with given
      $\alpha \in \left[0, \frac{\sqrt{\kappa^2 + 2\kappa +4}-\kappa}{2+ \kappa} \right]$,
      if $\kappa=\frac{\eta}{L}$ we have
\[
\frac{f(x^2)-f^\star}{f(x^1)-f^\star}\leq \beta < 1,
\]
where
\begin{equation}
\label{def:beta}
    \beta := \kappa \alpha^2 - \frac{\kappa(2-2\kappa)}{\kappa+2}\alpha + \frac{2-2\kappa}{2+\kappa}.
\end{equation}

\end{theorem}
\begin{proof}
  We assume w.l.o.g.\ that $f^\star=0$ and $x^1=0$, so that $f_1(x^1) = 0$. This implies that $x^2 = y^1 + \alpha(y^1-x^1) = (1+\alpha)y^1$.
  Moreover, the optimality conditions for the DCA subproblem that yields $y^1$ imply
  \[
  g_1(y^1) = -g_2(x^1) \Longleftrightarrow Ly^1 = - g_2(x^1).
  \]
Thus $x^2 = - \frac{1+\alpha}{L}g_2(x^1)$, and therefore
\begin{equation}
\label{eq:f1x2}
f_1(x^2) = \frac{L}{2}\left\|\frac{1+\alpha}{L}g_2(x^1)\right\|^2.
\end{equation}
   Our proof relies on the following identity, that may be verified through direct calculation:
\begin{align}
&{\left(\frac{L}{2}\left\|\frac{1+\alpha}{L}g_2(x^1)\right\|^2  -f_2(x^2)-f^\star\right)}-\beta\left( {-f_2(x^1)-f^\star} \right) \label{i} \\
&+\alpha \left( f_2(x^1)-f_2(y^{1})+\left\langle g_2(y^{1}), \frac{1}{L}g_2(x^{1})\right\rangle-\tfrac{1}{4L}\left\|g_2(x^1)-g_2(y^{1})\right\|^2\right) \label{ib}\\
&+ \left(\frac{(2-\kappa)}{\kappa+2}\alpha + \frac{2}{\kappa+2}\right) \left( f_2(y^1)-f_2(x^{1})-\left\langle g_2(x^{1}), \frac{1}{L}g_2(x^{1})\right\rangle-\tfrac{1}{4L}\left\|g_2(x^1)-g_2(y^{1})\right\|^2\right) \label{ii}\\
&+\left( f_2(x^2)-f_2(y^{1})-\left\langle g_2(y^{1}), \frac{1+\alpha}{L}g_2(x^1)-\frac{1}{L}g_2(x^{1})\right\rangle-\tfrac{1}{4L}\left\|g_2(x^2)-g_2(y^{1})\right\|^2\right) \label{iii}\\
&+\left(-\kappa\alpha^2 - \frac{2\kappa^2}{\kappa+2}\alpha + \frac{2\kappa}{\kappa+2}\right)\left(\frac{1}{2\eta}\left\|g_2(x^1)\right\|^2+f_2(x^1)\right) \label{iv}\\
&+\left(\frac{\kappa(2\alpha+1)}{\kappa+2}\right)\left(\frac{1}{2\eta}\left\|g_2(x^1)-g_2(y^1)\right\|^2-\frac{L}{2}\left\|\frac{1}{L}g_2(x^1)\right\|^2+f_2(y^1)\right) \label{v}\\
&=-\frac{1}{4L}\|g_2(x^2)-g_2(y^1)\|^2. \nonumber
\end{align}
By \eqref{eq:f1x2}, and since $f_1(x^1) = 0$, the first line \eqref{i} in the identity may be rewritten as:
$$
f(x^2)-f^\star- \beta\left(f(x^1)-f^\star  \right) = \frac{L}{2}\left\|\frac{1+\alpha}{L}g_2(x^1)\right\|^2  -f_2(x^2)-f^\star -\beta\left( {-f_2(x^1)-f^\star} \right).
$$
We will argue that all the terms \eqref{ib}, \eqref{ii}, \eqref{iii}, \eqref{iv}, and \eqref{v} are all nonnegative.
The identity will then imply $$f(x^2)-f^\star- \beta\left(f(x^1)-f^\star  \right) \le -\frac{1}{4L}\|g_2(x^2)-g_2(y^1)\|^2 \le 0,$$
so that the outer inequality yields the required result.

To this end, note that \eqref{ib} corresponds the interpolation constraint for $f_2$ at $(u,v) = (x^1,y^1)$; more precisely, \eqref{ib} equals
 the quantity
$- \alpha \mathcal{Q}_2(x^1,y^1)$ in \eqref{eq:interpolation constraints}  for $(u,v) = (x^1,y^1)$. Since the interpolation conditions require
$\mathcal{Q}_2(x^1,y^1) \le 0$, one has $- \alpha \mathcal{Q}_2(x^1,y^1) \ge 0$, as required.
Similarly, \eqref{ii} corresponds the interpolation constraint for $f_2$ at $(u,v) = (y^1,x^1)$ with
nonnegative multiplier $\left(\frac{(2-\kappa)}{\kappa+2}\alpha + \frac{2}{\kappa+2}\right)$,
and \eqref{iii} to  the interpolation constraint for $f_2$ at $(u,v) = (x^2,y^1)$.

Line \eqref{iv} corresponds to the P{\L} inequality \eqref{PL constraint} at $x^1$, and the multiplier
 $\left(-\kappa\alpha^2 - \frac{2\kappa^2}{\kappa+2}\alpha + \frac{2\kappa}{\kappa+2}\right)$
is nonnegative precisely if $\alpha$ is upper bounded as in the statement of the theorem.
Similarly, \eqref{v} corresponds to the P{\L} inequality \eqref{PL constraint} at $y^1$ with a nonnegative multiplier.
This completes the proof.

\end{proof}

{ Note that, under the conditions of the theorem, one has $f(x^2) \le f(x^1)$. As a consequence, all iterates lie in the set $X=\{x: f(x)\leq f(x^1)\}$  where the P{\L} inequality holds. Thus BCDA converges at a linear rate with constant $\beta < 1$ as defined in \eqref{def:beta}.}

Note that the right hand side is a convex parabola in $\alpha$ with minimum at  $\alpha^* = \frac{1-\kappa}{2+\kappa}$.
We immediately obtain the following corollary for gradient descent.

\begin{corollary}\label{cor:GD PL opt}
Let $f\in\mathcal{F}_{-L, L}({\mathbb{R}^n})$ satisfy the P{\L} inequality on $\{x: f(x)\leq f(x^1)\}$ with modulus $\eta$, then
\[
x^2 = x^1 - \frac{1 + \alpha^*}{L}\nabla f(x^1)
\]
satisfies
\[
\frac{f(x^2)-f^\star}{f(x^1)-f^\star}\leq \frac{4-\kappa^3 - 3\kappa}{(2+\kappa)^2},
\]
for $\alpha^* = \frac{1-\kappa}{2+\kappa}$ and $\kappa=\frac{\eta,}{L}$.
\end{corollary}
The optimal choice for the boosting parameter $\alpha^* = \frac{1-\kappa}{2+\kappa}$ corresponds to the step length
\[
\frac{1+\alpha^*}{L} = \frac{3}{(2+\kappa)L}.
\]
This leads to the convergence bound
\[
\frac{f(x^2)-f^*}{f(x^1)-f^*} \le \frac{4-\kappa^3-3\kappa}{(2+\kappa)^2}.
\]
This step length improves on the previous best known step length, given by $\min\left\{\frac{3}{2L}, \frac{2}{(1+\sqrt{\kappa})L}\right\}$,
 from \cite[Corollary 1]{abbaszadehpeivasti2022PL}
 for {$L$-smooth convex functions}.
 A direct comparison shows the new bound is tighter if and only if the condition number is in the range $0 < \kappa < 1/4$.

Figure \ref{fig:4} illustrates the bound from Corollary~\ref{cor:GD PL opt}.
The figure shows that the bound from Corollary~\ref{cor:GD PL opt} improves on the earlier bound from \cite[Theorem 3(ii)]{abbaszadehpeivasti2022PL} for the same step length.
To show that the new bound is still not tight, we include another (numerically computed) bound in the figure that uses the stronger characterization of smooth \L{}ojasiewicz functions from~\cite[Proposition 3.4]{rubbens2025constructive}.
We were unable, though, to obtain an analytical expression for the latter bound, and therefore omit further details here, especially since we believe the
numerically computed bound is still not tight.
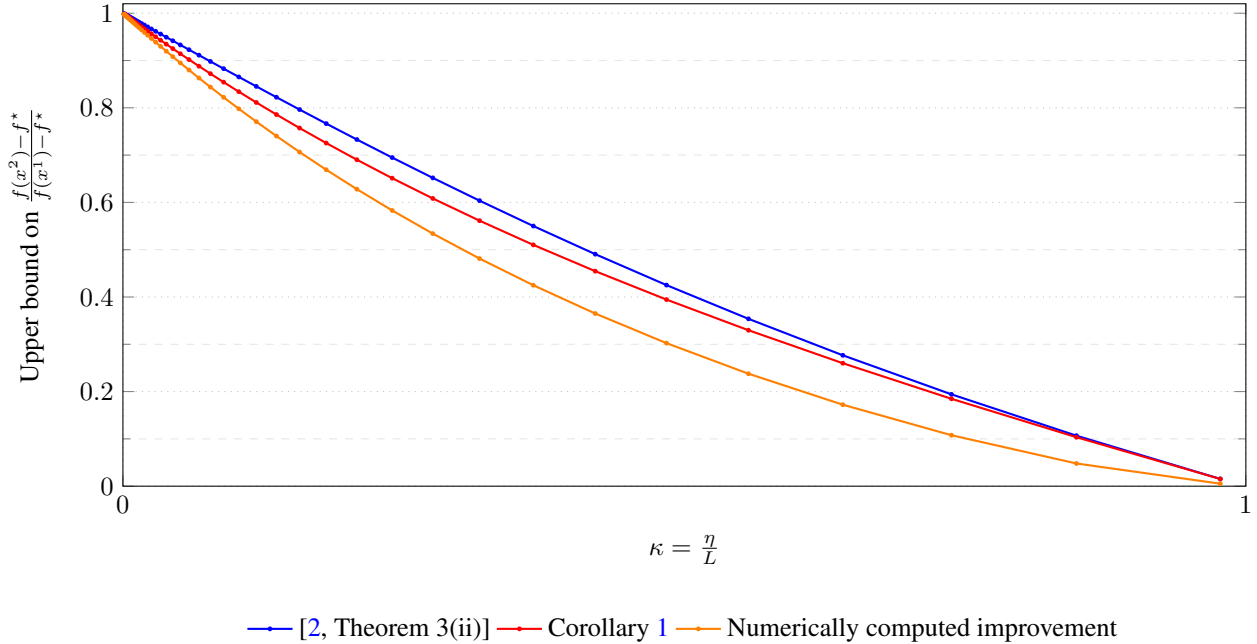
\begin{figure}[h!]
\centering
\begin{tikzpicture}
\begin{axis}[
    width=1\textwidth,
    height=8cm,
    ylabel={Upper bound on $\frac{f(x^2)-f^\star}{f(x^1)-f^\star}$},
    xlabel={$\kappa=\frac{\eta}{L}$},
    xmin=0., xmax=1,
    ymin=0, ymax=1.02,
    xtick={0,...,1.},
    yticklabel style={/pgf/number format/fixed, /pgf/number format/precision=2},
    grid=both,
    major grid style={dotted,gray!50},
    minor grid style={dashed,gray!20},
    minor x tick num=0,
    minor y tick num=1,
    legend style={
        at={(0.5,-0.25)}, % Below the plot
        anchor=north,
        legend columns=3,
        draw=none,
        fill=none,
        cells={align=left}
    }
]

% === Plot each alpha's CSV data with legend ===
\addplot[
    thick, color=blue, mark=*, mark options={solid}, mark size=.5
] table[x=eta,y=GD_OLD_CONV,col sep=comma] {Data/GD_Loja/GD_Loja.csv};
\addlegendentry{\cite[Theorem 3(ii)]{abbaszadehpeivasti2022PL}}
\addplot[
    thick, color=red, mark=*, mark options={solid}, mark size=.5
] table[x=eta,y=GD_NEW_CONV,col sep=comma] {Data/GD_Loja/GD_Loja.csv};
\addlegendentry{Corollary~\ref{cor:GD PL opt}}
\addplot[
    thick, color=orange, mark=*, mark options={solid}, mark size=.5
] table[x=eta,y=GD_Numerics,col sep=comma] {Data/GD_Loja/GD_Loja.csv};
\addlegendentry{Numerically computed improvement}
\end{axis}
\end{tikzpicture}
\caption{\label{fig:4} Comparisons between the bounds from Corollary~\ref{cor:GD PL opt} on
 $\frac{f(x^2)-f^\star}{f(x^1)-f^\star}$ for gradient descent with step-size $\frac{3}{(2+\kappa ) L}$ for a range of $\kappa$.
 The `numerically computed improvement' uses the stronger characterization of
 smooth \L{}ojasiewicz functions from~\cite[Proposition 3.4]{rubbens2025constructive}.}
\end{figure}

\section{Conclusion and discussion}
\label{sec:conclusion}
In this paper we extended the SDP performance analysis for DCA by Abbaszadeh et al.\ \cite{abbaszad2024DCA} to the boosted DCA.
Having said that, our main result in Theorem \ref{thm:main} only covers the case where both $f_1,f_2 \in \mathcal{F}_{\mu,L}(\mathbb{R}^n)$
for given $0 < \mu < L$. It would be interesting to obtain more detailed results with $f_1 \in \mathcal{F}_{\mu_1,L_1}(\mathbb{R}^n)$
and $f_2 \in \mathcal{F}_{\mu_2,L_2}(\mathbb{R}^n)$ where $\mu_1 \ne \mu_2$, etc, and where (one of) $L_1$ and $L_2$ may be infinite.
(Such results were presented in \cite{abbaszad2024DCA}.)

In related work, Rotaru et al.\ \cite{rotaru2025tight} extended the  analysis in \cite{abbaszad2024DCA} to generalised DCA,
where one allows $\mu_2 < 0$,
as long as $\mu_1+\mu_2 \ge 0$. (It is known that this generalised DCA is still a descent method.)
In the same work, Rotaru et al.\ \cite{rotaru2025tight} also consider the implications for the special case of the proximal gradient method,
 and obtain some improved
worst-case convergence bounds. This is similar in spirit to the analysis in our paper, where we derived new results for  the special case of gradient descent (Theorem \ref{thm:GD PL}).
A natural question is to extend the analysis by Rotaru et al.\ \cite{rotaru2025tight} to the boosted case.

This work raises a few open questions:
\begin{enumerate}
  \item
  In Theorem \ref{thm:main} and Theorem \ref{thm:GD PL}, the performance measure is a convex function of the boosted step length $\alpha$.
  We also observe this numerically in Figure \ref{fig:3}, even though this is not covered by theory. A natural question is whether this is always the case.
  \item
  Condition~\eqref{PL constraint} was used to model the P{\L} inequality. This condition is known to
  be necessary, but not sufficient, to characterize the desired function class in the performance analysis framework.
   While Section~\ref{sec:GD} shows that adding appropriate points to the discrete description of the function at hand
   allows improving upon known convergence results, it remains clear that~\eqref{PL constraint} can also
    be refined. A natural (but complicated) candidate is provided in~\cite[Proposition 3.4]{rubbens2025constructive}
    and is shown numerically to allow improving even more the results from Theorem~\ref{thm:GD PL} and Corollary~\ref{cor:GD PL opt}
     (using the refined discrete description of the function); see Figure \ref{fig:4}. The results from Theorem~\ref{thm:GD PL} and Corollary~\ref{cor:GD PL opt} are therefore not tight in general, and it remains unclear how to obtain analytical expressions of tight bounds.
  \item As for now, it remains open to generalise the analysis to possibly include constraints,
  e.g., by allowing one of the Lipschitz constants to take the value $+\infty$. This is of practical interest, though, since the boosted DCA has been extended to allow linear constraints; see \cite{Boosted_DCA_linear_constraints,Zhang_et_al_2024_JOTA}. Any progress in this direction would be welcome.
  \item 
 { It would be desirable to derive a generalization of Theorem~\ref{thm:main} where  $f_1\in\mathcal{F}_{\mu_1, L_1}({\mathbb{R}^n})$ and $f_2\in\mathcal{F}_{\mu_2, L_2}({\mathbb{R}^n})$, and where one allows $\mu_1 \ne \mu_2$ and $L_1 \neq L_2$.
However, we could only establish Theorem \ref{thm:main} for the case where $\mu_1 = \mu_2 = \mu > 0$ and $L_1 = L_2 = L > \mu$. Theorem 3.1 in \cite{abbaszad2024DCA}, on the other hand, also establishes worst-case bounds for some cases of DCA where $\mu_1 \neq \mu_2$, $L_1 \neq L_2$, and where $L_1$ or $L_2$ may  not be finite, and $\mu_1$ or $\mu_2$ may be zero. Thus, the new results in this paper do not imply all the known results for DCA from \cite{abbaszad2024DCA} by setting $\alpha = 0$. 
\item 
In Example \ref{ex:tight} we showed that the upper bound in Theorem \ref{thm:main} can be tight for specific values of the parameters $\mu$, $L$, $\alpha$, and $N$.
We conjecture though, that the bound is in fact tight for all values of these parameters that meet the conditions in 
Theorem \ref{thm:main}. In fact, we believe that the optimal value of the SDP problem \eqref{PEP} is precisely the upper bound in Theorem \ref{thm:main} times $L$, when $\mu_1 = \mu_2 = \mu$, and $L_1 = L_2 = L$. This conjecture is based on extensive numerical experiments, as well as a general construction given in the appendix to this paper. This construction was used to generate Example \ref{ex:tight}, but we believe --- although we cannot prove this --- that it provides a tight bound for all choices of the parameters.
In particular, we conjecture that the general construction in the appendix gives a rank-one optimal solution to the SDP problem \eqref{PEP}. 
}
\end{enumerate}

\subsection*{Data availability statement}
There are no external data associated with this manuscript. The SDP performance problems we considered in this paper have been implemented in the Matlab software PESTO~\cite{pesto2017} and the Python version PEPit~\cite{pepit2024}. Codes allowing the reproduction of the main plots of this work can be found at
\begin{center}
\url{https://github.com/PerformanceEstimation/Boosted-DCA}
\end{center}

\subsection*{Compliance with Ethical Standards}
This research did not involve Human Participants and/or Animals, and there are no conflicts of interest to report.

\subsection*{Funding acknowledgement}
The work of the first two authors was supported by the
 Dutch Scientific Council (NWO)  grant \emph{Optimization for and with Machine Learning}, OCENW.GROOT.2019.015. A.~Taylor is supported by the European
Union (ERC grant CASPER 101162889). The French government also partly funded this work under the management
 of Agence Nationale de la Recherche as part of the “France 2030” program, reference ANR-23-IACL-0008 (PR[AI]RIE-PSAI). 
 Views and opinions expressed are however those of the authors only.

\section*{Appendix}
{
Here, we give a general construction that yields Example \ref{ex:tight} as the special case where $N=1$, $\mu =1$, $L =2$, and $\alpha =1$. After presenting this construction, we will give an indication of how it is derived.

Fix parameters
\[
0<\mu<L,\qquad N\in\mathbb{N},\qquad 0<\alpha\le \min\{1,2\mu/L\},\qquad \kappa:=\mu/L.
\]
Define
\[
D:=(1+\kappa\alpha)N+\frac{1}{2(1-\kappa)},\qquad
s:=\sqrt{\frac{L}{D}},\qquad
\delta:=\frac{s}{L},\qquad
q:=1+\kappa\alpha.
\]
Set
\[
x^1:=0,\qquad x^k:=-(k-1)(1+\alpha)\delta\quad (k=1,\dots,N+1),
\]
\[
y^k:=x^k-\delta,\qquad z^k:=x^k-\alpha\delta\quad (k=1,\dots,N),
\]
\[
x^\star:=x^{N+1}-\frac{s}{L-\mu}=x^{N+1}-\frac{\delta}{1-\kappa}.
\]

First define endpoint slopes:
\[
g_2(x^k):=-(k-1)qs\quad (k=1,\dots,N+1),
\qquad
g_2(x^\star):=g_2(x^{N+1})+\mu(x^\star-x^{N+1}).
\]
Define $g_2:\R\to\R$ by
\[
g_2(x):=
\begin{cases}
g_2(x^\star)+\mu(x-x^\star), & x\le x^\star,\\[1mm]
g_2(x^{N+1})+\mu(x-x^{N+1}), & x^\star\le x\le x^{N+1},\\[1mm]
g_2(x^{k+1})+L(x-x^{k+1}), & x^{k+1}\le x\le z^k,\ \ k=1,\dots,N,\\[1mm]
g_2(x^{k})+\mu(x-x^{k}), & z^k\le x\le x^k,\ \ k=1,\dots,N,\\[1mm]
Lx, & x\ge 0.
\end{cases}
\]
Now define $f_2$ by integrating $g_2$ from $x^\star$:
\[
f_2(x):=\int_{x^\star}^{x} g_2(t)\,dt,
\quad\text{so that } f_2(x^\star)=0.
\]

Prescribe
\[
g_1(x^k):=g_2(x^k)+s\quad (k=1,\dots,N+1),
\qquad
g_1(y^k):=g_2(x^k)\quad (k=1,\dots,N).
\]
Define $g_1:\R\to\R$ by
\[
g_1(x):=
\begin{cases}
g_2(x^\star)+\mu(x-x^\star), & x\le x^\star,\\[1mm]
g_1(x^{N+1})+L(x-x^{N+1}), & x^\star\le x\le x^{N+1},\\[1mm]
g_1(x^{k+1})+\mu(x-x^{k+1}), & x^{k+1}\le x\le y^k,\ \ k=1,\dots,N,\\[1mm]
g_1(y^{k})+L(x-y^{k}), & y^k\le x\le x^k,\ \ k=1,\dots,N,\\[1mm]
s+Lx, & x\ge 0.
\end{cases}
\]
Now define $f_1$ by integrating $g_1$ from $x^\star$:
\[
f_1(x):=\int_{x^\star}^{x} g_1(t)\,dt,
\quad\text{so that } f_1(x^\star)=0.
\]

\subsection*{Explanation of the construction}
The construction above is based on 
observations of which constraints in the SDP problem \eqref{PEP} are binding at optimality. One may then attempt to construct a solution to these equations.
Below we mention the relevant equalities.

Assume $f=f_1-f_2$, $f^\star=0$, $f_1,f_2\in\mathcal{F}_{\mu,L}(\mathbb{R}^n)$, $0<\mu<L<\infty$, and define
\[
s_k := g_1(x^k)-g_2(x^k).
\]

\subsection*{Algorithmic equalities (for $k=1,\dots,N$)}
\[
g_1(y^k)=g_2(x^k),
\qquad
x^{k+1}=y^k+\alpha\,(y^k-x^k).
\]

\subsection*{Interpolation equalities used in Table 1 (for $k=1,\dots,N$)}
\[
Q_1(x^k,y^k)=0,\quad Q_1(y^k,x^k)=0,\quad Q_1(x^{k+1},y^k)=0,\quad Q_1(y^k,x^{k+1})=0,
\]
\[
Q_2(x^{k+1},x^k)=0,
\]
where for $\ell\in\{1,2\}$,
\[
\begin{aligned}
Q_\ell(u,v):={}&
\frac{1}{2\left(1-\frac{\mu}{L}\right)}
\Bigg(
\frac{1}{L}\bigl\|g_\ell(u)-g_\ell(v)\bigr\|^2
+\mu\|u-v\|^2
-\frac{2\mu}{L}\left\langle g_\ell(v)-g_\ell(u),\,v-u\right\rangle
\Bigg)\\
&- f_{\ell}(u)+f_{\ell}(v)+\left\langle g_{\ell}(v),\, u-v\right\rangle .
\end{aligned}
\]

\subsection*{Vanishing squared-norm terms (for $k=1,\dots,N$)}
\[
g_1(x^k)-g_2(x^k)-L(x^k-y^k)=0,
\]
\[
-g_2(x^k)+g_2(x^{k+1})+(L+\alpha\mu)(x^k-y^k)=0,
\]
\[
\begin{aligned}
(\alpha  \mu -2 (\alpha -1) L)g_1(x^{k+1})
&+L(\alpha-2)g_2(x^k)+\alpha(L-\mu)g_2(x^{k+1})\\
&+\alpha  L (-\alpha  \mu +\mu +L)(x^k-y^k)=0.
\end{aligned}
\]

\subsection*{One-step inequality as equality (for $k=1,\dots,N$)}
\[
f(x^{k+1})+\frac{1}{2L}\|s_{k+1}\|^2
+\frac{1}{L}\left(\frac12+\alpha\frac{\mu}{L}\right)\|s_k\|^2
= f(x^k).
\]

\subsection*{Final descent-lemma equality at $x^{N+1}$}
\[
f(x^{N+1})=\frac{1}{2(L-\mu)}\|s_{N+1}\|^2.
\]

\subsection*{Initial gap}
\[
f(x^1)=1.
\]

}
\end{document}